\documentclass[11pt]{amsart}
\usepackage{graphicx}
\usepackage{amsmath} 
\usepackage{amsthm,amsfonts,amssymb,mathrsfs,amscd,amstext,amsbsy} 
\usepackage{epic,eepic} 
\usepackage{yfonts}
\usepackage{paralist,enumerate}
\usepackage[all]{xy}
\usepackage{hyperref}
\usepackage{multirow}
\hypersetup{colorlinks}

\newtheorem{theorem}{Theorem}[section]
\newtheorem{cor}[theorem]{Corollary}
\newtheorem{lemma}[theorem]{Lemma}
\newtheorem{conj}[theorem]{Conjecture}

\newtheorem{lem}[theorem]{Lemma}

\newtheorem{rem}[theorem]{Remark}

\newtheorem*{Definition*}{Definition}

\def\qed{\hfill \ifhmode\unskip\nobreak\fi\quad\ifmmode\Box\else$\Box$\fi\\ }

\begin{document}

\title[Almost complex $S^1$-manifolds with isolated fixed points]{Circle actions on almost complex manifolds with isolated fixed points}
\author{Donghoon Jang}
\address{School of Mathematics, Korea Institute for Advanced Study, 85 Hoegiro, Dongdaemun-gu, Seoul, 02455, Korea}
\email{groupaction@kias.re.kr}
\begin{abstract}
In \cite{J1}, the author proves that if the circle acts symplectically on a compact, connected symplectic manifold $M$ with three fixed points, then $M$ is equivariantly symplectomorphic to some standard action on $\mathbb{CP}^2$. In this paper, we extend the result to a circle action on an almost complex manifold; if the circle acts on a compact, connected almost complex manifold $M$ with exactly three fixed points, then $\dim M=4$. Moreover, the weights at the fixed points agree with those of a standard circle action on the complex projective plane $\mathbb{CP}^2$. Also, we deal with the cases of one fixed point and two fixed points.
\end{abstract}
\maketitle


\section{Introduction}

The purpose of this paper is to classify circle actions on compact almost complex manifolds with few fixed points. In \cite{J1}, the author classifies a symplectic circle action on a compact symplectic manifold with three fixed points:

\begin{theorem} \label{t11} \cite{J1} Let the circle act symplectically on a compact, connected symplectic manifold $M$. If there are exactly three fixed points, then $M$ is equivariantly symplectomorphic to some standard action on $\mathbb{CP}^2$ and the weights at the fixed points are $\{a+b,a\},\{-a,b\}$, and $\{-b,-a-b\}$ for some positive integers $a$ and $b$. \end{theorem}

In particular, the manifold has to be four-dimensional and the action must be Hamiltonian. In this paper, we extend the result to a circle action on an almost complex manifold\footnote{Throughout the paper, we assume the action preserves the almost complex structure.}. The main result of this paper is the classification of a circle action on a compact almost complex manifold with at most three fixed points.

\begin{theorem} \label{t12} Let the circle act on a compact, connected almost complex manifold $M$. 
\begin{enumerate}
\item If there is exactly one fixed point, then $M$ is a point. 
\item If there are exactly two fixed points, then either $\dim M=2$ or $\dim M=6$. If $\dim M=2$, $M$ is the 2-sphere and the weights at the fixed points are $\{a\}$ and $\{-a\}$ for some some positive integer $a$. If $\dim M=6$, then the weights at the fixed points are $\{-a-b,a,b\}$ and $\{-a,-b,a+b\}$ for some positive integers $a$ and $b$. 
\item  If there are exactly three fixed points, then $\dim M=4$. Moreover, the weights at the fixed points are $\{a+b,a\},\{-a,b\}$, and $\{-b,-a-b\}$ for some positive integers $a$ and $b$.
\end{enumerate}
\end{theorem}
Theorem \ref{t12} will follow immediately from Theorem \ref{t26} and Theorem \ref{t29}.

We compare circle actions with few fixed points on almost complex manifolds, symplectic manifolds, and complex manifolds by providing a table. For this, assume that if the circle acts on an almost complex (symplectic, and complex) manifold $M$, then the action preserves the almost complex (symplectic, and complex, respectively) structure. 

Since any symplectic or complex manifold is almost complex, Theorem \ref{t12} implies the same results on the dimension and the weights at the fixed points for any symplectic or complex manifold as for an almost complex manifold. Moreover, since the weights at the fixed points determine the Chern (and Pontryagin) numbers and the Hirzebruch $\chi_y$-genus, those invariants are the same for the three types of manifolds. On the other hand, on the existence whether such a manifold exists or not and on the uniqueness if we can determine such a manifold up to diffeomorphism (symplectomorphism or biholomorphism, respectively), the answers depend on the type of the manifold. In Table \ref{t1}, any manifold is compact and connected, any circle action on the manifold preserves the given structure, and dim denotes the dimension of the manifold $M$. To the author's best knowledge, the classification is as in Table \ref{t1}. The more there are fixed points, the harder the classification problem is; for the classification of an almost complex $S^1$-manifold with four fixed points in low dimensions, see \cite{J3}.

\begin{table}
\begin{center}
\begin{tabular}{|p{0.1cm}|p{0.9cm}|p{1cm}|p{2cm}|p{2cm}|p{2cm}|}
\hline
\multicolumn{3}{|c|}{number of fixed points} & \multicolumn{3}{|c|}{type of manifold} \\\cline{4-6}
\multicolumn{3}{|c|}{} & almost complex & symplectic & complex \\\hline
\multicolumn{3}{|c|}{1} & \multicolumn{3}{|c|}{$M=\{\textrm{pt}\}$} \\\hline
\multirow{4}{*}{2} & \multicolumn{2}{|c|}{2 dim} & \multicolumn{3}{|c|}{$M=S^2$} \\\cline{2-6}
{} & \multirow{3}{*}{6 dim} & invariants & \multicolumn{3}{|p{8cm}|}{Chern (and Pontryagin) numbers and the Hirzebruch $\chi_y$-genus are the same as $S^6$.} \\\cline{3-6}
{} & {} & existence & $S^6$ & not known & not known \\\cline{3-6}
{} & {} & uniqueness & \multicolumn{3}{|c|}{not known} \\\hline
\multirow{3}{*}{3} & \multirow{3}{*}{4 dim} & invariants & \multicolumn{3}{|p{8cm}|}{Chern (and Pontryagin) numbers and the Hirzebruch $\chi_y$-genus are the same as $\mathbb{CP}^2$.} \\\cline{3-6}
{} & {} & existence & \multicolumn{3}{|c|}{$\mathbb{CP}^2$} \\\cline{3-6}
{} & {} & uniqueness & not known & \multicolumn{1}{|p{2.5cm}|}{$M$ is equivariantly symplectomorphic to $\mathbb{CP}^2$.} & \multicolumn{1}{|p{2.5cm}|}{Possibly $M$ is biholomorphic to $\mathbb{CP}^2$ by \cite{CHK}.} \\\hline
\end{tabular}
\caption{The classification of $S^1$-manifolds with few fixed points.}
\label{t1}
\end{center}
\end{table}

Now, we discuss the proof of Theorem \ref{t12}. When there are one or two fixed points, then we give a complete proof in Section 2; see Theorem \ref{t26}. For the case of three fixed points (Theorem \ref{t29}), the idea of the proof is to adapt the proof of Theorem \ref{t11} in \cite{J1}, since basically the same proof applies. To prove Theorem \ref{t11}, in \cite{J1} the author uses the symplectic property in a number of places, and we carefully go through an eliminating this reliance. In particular, \cite{J1} adapts results for symplectic actions from other papers. We will have extensions of the results to almost complex $S^1$-manifolds in Section 2. If we extend to almost complex manifolds all the results that we need, then the same proof as in \cite{J1} goes through. Therefore, we conclude that a compact almost complex $S^1$-manifold $M$ with three fixed points must have $\dim M=4$ and the weights as described in Theorem \ref{t12}. We clarify this at the end of this paper.

For a circle action on an almost complex manifold, there is an interesting and important conjecture by Kosniowski on the relationship between the dimension of a manifold and the number of fixed points \cite{K2}. 

\begin{conj} \cite{K2}
Let the circle act on a $2n$-dimensional compact, connected almost complex manifold with $k$ fixed points. Then $n \leq f(k)$, where $f(k)$ is a linear function in $k$.
\end{conj}

Kosniowski conjectures further that $f(k)=2k$. Theorem \ref{t12} confirms that the conjecture is true if the number of fixed points is at most three. In general, Kosniowski's conjecture is challenging. Nevertheless, we verify the conjecture in a special case. For this, let the circle act on a compact almost complex manifold $M$ with isolated fixed points. The \emph{Chern class map} of $M$ is the map
\begin{center}
$c_1(M):M^{S^1} \longrightarrow \mathbb{Z}, p \mapsto c_1(M)(p) \in \mathbb{Z}$,
\end{center}
where $c_1(M)(p)$ is the first equivariant Chern class $c_1(M)$ at $p$. A map $f:X\longrightarrow Y$ between sets is called \emph{somewhere injective} if there exists an element $y$ in $Y$ such that $f^{-1}(\{y\})$ is the singleton. In \cite{PT}, Pelayo and Tolman prove that if the Chern class map of a symplectic $S^1$-action on a compact symplectic manifold $M$ is somewhere injective, then the action has at least $\frac{1}{2}\dim M+1$ fixed points. The result naturally extends to almost complex $S^1$-manifolds.

\begin{theorem} \label{t27} Let the circle act on a compact almost complex manifold $M$. If the Chern class map is somewhere injective, then the action has at least $\frac{1}{2}\dim M+1$ fixed points. \end{theorem}

In \cite{L2}, Li makes a conjecture concerning circle actions on orientable manifolds with three fixed points, and the statement is as follows:

\begin{conj} \cite{L2} \label{c1}
A closed, smooth and orientable $S^1$-manifold $M$ with $\dim M \neq 4$ cannot have exactly three fixed points. \end{conj}

Given a closed orientable $S^1$-manifold $M$ with only isolated fixed points, the total number of weights over all the fixed points must be even. In particular, this implies that if there is an odd number of fixed points, then $\dim M$ must be a multiple of 4; for a proof, see Corollary 2.7 of \cite{J4}. Therefore, Conjecture \ref{c1} is obvious if $\dim M$ is not a multiple of 4. 

Theorem \ref{t12} proves Conjecture \ref{c1} in almost complex case, i.e., when the orientable manifold $M$ admits an almost complex structure and the circle action preserves the almost complex structure. In the general case, to the author's knowledge, the answer to the conjecture is unknown.

Finally, the author would like to thank the anonymous referee for valuable comments.

\section{Properties and Classification}

To classify circle actions on almost complex manifolds with few fixed points, we shall introduce properties that every almost complex $S^1$-manifold with only isolated fixed points must satisfy. To prove Theorem \ref{t12}, we need the following basic facts for a circle action on an almost complex manifold.

\begin{enumerate}[(1)]
\item At each isolated fixed point $p$, there are well-defined non-zero integers $w_p^i$, called \emph{weights}, for $1 \leq i \leq n$, where $\dim M=2n$.

\item Let $k$ be an integer such that $k>1$. As a subgroup of $S^1$, $\mathbb{Z}_k$ also acts on $M$. The set $M^{\mathbb{Z}_k}$ of points fixed by the $\mathbb{Z}_k$-action is a union of smaller dimensional almost complex submanifolds, called isotropy submanifolds\footnote{Here we assume that the action is effective.}. (In the symplectic case, this is a union of smaller dimensional symplectic submanifolds.)
\end{enumerate}

A symplectic circle action on a symplectic manifold is a particular case of a circle action on an almost complex manifold. If the circle acts symplectically on a symplectic manifold, the set of almost complex structures compatible with the symplectic form is contractible. Therefore, the weights at each isolated fixed point are well-defined.

Let $M$ be a compact almost complex manifold. The Hirzebruch $\chi_y$-genus $\chi_y(M)$ of $M$ is the genus belonging to the power series $\frac{x(1+ye^{-x(1+y)})}{1+e^{-x(1+y)}}$. Suppose that the circle acts on $M$ with isolated fixed points. For any $t \in S^1$, we can define an equivariant index of the Dolbeault-type operator. In \cite{L}, Li proves that the equivariant index of the operator is rigid under the circle action; as a consequence Li obtains the following formulae.

\begin{theorem} \cite{L} \label{t21}
Let the circle act on a $2n$-dimensional compact almost complex manifold $M$ with isolated fixed points. For each $i$ such that $0 \leq i \leq n$,
\begin{center}
$\displaystyle \chi^i(M)=\sum_{p \in M^{S^1}} \frac{\sigma_i (t^{w_p^1}, \cdots, t^{w_p^n})}{\prod_{j=1}^n (1-t^{w_p^j})} = (-1)^i N^i = (-1)^i N^{n-i}$,
\end{center}
where $\chi_y(M)=\sum_{i=0}^n \chi^i(M) \cdot y^i$ is the Hirzebruch $\chi_y$-genus of $M$, $t$ is an indeterminate, $\sigma_i$ is the $i$-th elementary symmetric polynomial in $n$ variables, and $N^i$ is the number of fixed points with exactly $i$ negative weights.
\end{theorem}

Given a compact almost complex $S^1$-manifold with isolated fixed points, the weights at the fixed points satisfy various properties.

\begin{lem} \label{l22} \cite{H}, \cite{L}
Let the circle act on a $2n$-dimensional compact almost complex manifold $M$ with isolated fixed points. Then $N^i=N^{n-i}$ for all $i$, where $N^i$ is the number of fixed points with exactly $i$ negative weights. \end{lem}

The following is an immediate consequence of Lemma \ref{l22}.

\begin{cor} \label{c25} 
Let the circle act on a $2n$-dimensional compact almost complex manifold with $k$ fixed points. If $k$ is odd, then $n$ is even. \end{cor}

\begin{lem} \label{l23} \cite{H}, \cite{L}
Let the circle act on a compact almost complex manifold $M$ with isolated fixed points. For each $w \in \mathbb{Z}$,
\begin{center}
$\displaystyle \sum_{p \in M^{S^{1}}} N_{p}(w)=\sum_{p \in M^{S^{1}}} N_{p}(-w)$,
\end{center}
where $N_{p}(w)$ is the multiplicity of $w$ in the isotropy representation $T_{p} M$ for all $p \in M^{S^{1}}$.
\end{lem}

At each fixed point $p$, the first equivariant Chern class $c_1(M)$ at $p$ is equal to the sum of the weights at $p$ (times a generator of $H_{S^1}^*(\{\textrm{pt}\})$. Therefore, the following lemma is an immediate consequence of Lemma \ref{l23}.

\begin{theorem} \label{t24} \cite{H} If the circle acts on a compact almost complex manifold $M$ with isolated fixed points, then $\sum_{p \in M^{S^1}} c_1(M)(p)=0$, where $c_1(M)(p)$ is the first equivariant Chern class of $M$ at $p$. \end{theorem}

In \cite{T}, Tolman proves a property that the weights at the fixed points which lie in the same isotropy submanifold of a symplectic circle action satisfy. It naturally extends to an almost complex $S^1$-manifold and the statement appears in \cite{GS}.

\begin{lem} \label{l24} \cite{T}, \cite{GS}
Let the circle act on a compact almost complex manifold $M$. Let $p$ and $p'$ be fixed points which lie in the same component $N$ of $M^{\mathbb{Z}_{k}}$, for some $k>1$. Then the $S^{1}$-weights at $p$ and at $p'$ are equal modulo $k$. \end{lem}

In what follows we classify an almost complex $S^1$-manifold with one or two fixed points. In \cite{PT}, Pelayo and Tolman classifies a symplectic circle action on a compact symplectic manifold with two fixed points. Earlier than this, Kosniowski classifies a holomorphic vector field on a compact complex manifold with two simple isolated zeroes \cite{K}. Moreover, the eigenvalues at the two zeroes are classified. Kosniowski's idea is to utilize the index formula for holomorphic vector fields that is analogous to one in Theorem \ref{t21}. Since neither \cite{PT} nor \cite{K} are in the context of almost complex manifolds, we provide a separate proof. Closely following Kosniowski's idea, we utilize the index formula in Theorem \ref{t21} to classify a circle action on a compact almost complex manifold with one or two fixed points. First, we need the following technical lemma.

\begin{lemma} \cite{J2} \label{l25}
Let the circle act on a compact almost complex manifold $M$ with isolated fixed points such that $\dim M>0$. Then there exists $i$ such that $N^i \neq 0$ and $N^{i+1} \neq 0$, where $N^i$ is the number of fixed points with exactly $i$ negative weights. \end{lemma}

We can derive Lemma \ref{l25} from Proposition 1.2 in \cite{JT} by applying it to the smallest positive weight. With Theorem \ref{t21} and Lemma \ref{l25}, we classify an almost complex $S^1$-manifold with one or two fixed points.

\begin{theorem} \label{t26}
Let the circle act on a $2n$-dimensional compact, connected almost complex manifold $M$. Then there cannot be exactly one fixed point, unless $M$ is a point. If there are exactly two fixed points, then either $\dim M=2$ or $\dim M=6$. If $\dim M=2$, then the weights at the fixed points are $\{a\}$ and $\{-a\}$ for some positive integer $a$. If $\dim M=6$, then the weights at the fixed points are $\{-a-b,a,b\}$ and $\{-a,-b,a+b\}$ for some positive integers $a$ and $b$.
\end{theorem}

\begin{proof} 
First, suppose that there is exactly one fixed point. In this case, there are many proofs that $M$ must be the point itself. For instance, this follows from Theorem \ref{t21}; for any $i$, $\chi^i(M)$ must be a constant. However, if there is exactly one fixed point $p$, the expression $\frac{\sigma_i (t^{w_p^1}, \cdots, t^{w_p^n})}{\prod_{j=1}^n (1-t^{w_p^j})}$ cannot be constant for all $t \in S^1$. Moreover, Lemma \ref{l25} also gives the same conclusion.

Next, suppose that there are exactly two fixed points. Label the fixed points by $p$ and $q$. By Lemma \ref{l25}, there exists $i$ such that if one fixed point has exactly $i$ negative weights, then the other fixed point has exactly $i+1$ negative weights. Without loss of generality, assume that $p$ has exactly $i$ negative weights and hence $q$ has exactly $i+1$ negative weights. Hence $N^i=N^{i+1}=1$ and $N^j=0$ for $j \neq i, i+1$. By Theorem \ref{t21}, $N^i=N^{i+1}=N^{n-i}=1$. Therefore, $n=2i+1$ is odd and $i=\frac{1}{2}(n-1)$. 

Suppose that $n>3$. Then $N^0=N^1=0$ and hence by Theorem \ref{t21} we have
\begin{center}
$\displaystyle 0=N^0=\frac{1}{\prod_{j=1}^n (1-t^{w_p^j})}+\frac{1}{\prod_{j=1}^n (1-t^{w_q^j})}$

$\displaystyle 0=-N^1=\frac{\sum_{j=1}^n t^{w_p^j}}{\prod_{j=1}^n (1-t^{w_p^j})}+\frac{\sum_{j=1}^n t^{w_q^j}}{\prod_{j=1}^n (1-t^{w_q^j})}$.
\end{center}
This implies that $w_{p_1}^j$ and $w_{p_2}^j$ agree up to order, which is a contradiction since $p$ and $q$ have different numbers of negative weights (and hence different numbers of positive weights). Therefore, either $n=1$ or $n=3$.

If $n=1$, by Lemma \ref{l23}, there exists a positive integer $a$ so that the weight at $p$ is $a$ and the weights at $q$ is $-a$.

Suppose that $n=3$. Assume that $p$ has weights $-a_1,a_2,a_3$ and $q$ has weights $-b_1,-b_2,b_3$, where $a_i,b_i$ are positive integers. By Theorem \ref{t21}, we have

\begin{center}
$\displaystyle 0=N^0=\frac{1}{(1-t^{-a_1})(1-t^{a_2})(1-t^{a_3})}+\frac{1}{(1-t^{-b_1})(1-t^{-b_2})(1-t^{b_3})}$
\end{center}
and hence
\begin{center}
$\displaystyle 0=N^0=\frac{-t^{a_1}}{(1-t^{a_1})(1-t^{a_2})(1-t^{a_3})}+\frac{t^{b_1+b_2}}{(1-t^{b_1})(1-t^{b_2})(1-t^{b_3})}$.
\end{center}
This implies that $a_1=b_1+b_2$. Moreover, we have
\begin{center}
$(1-t^{b_1+b_2})(1-t^{a_2})(1-t^{a_3})=(1-t^{b_1})(1-t^{b_2})(1-t^{b_3})$.
\end{center}
This implies that $\{a_2,a_3\}=\{b_1,b_2\}$ and $b_3=a_2+a_3$. Let $a_2=a, a_3=b$. The result follows. \end{proof}

With Theorem \ref{t27}, Theorem \ref{t24}, and Theorem \ref{t26}, we obtain the following corollary.

\begin{cor} \label{c28}
Let the circle act on a compact almost complex manifold $M$ with non-empty fixed point set. Then there are at least two fixed points, and if $\dim M \geq 8$, then there are at least three fixed points. Moreover, if the Chern class map is not identically zero and $\dim M \geq 6$, then there are at least four fixed points. \end{cor}

\begin{proof} The first claim follows immediately from Theorem \ref{t26}. Next, suppose that $\dim M \geq 6$ and the Chern class map is not identically zero. Then the image of the Chern class map contains at least two elements. On the other hand, by Theorem \ref{t24}, $\sum_{p \in M^{S^1}} c_1(M)(p)=0$. Therefore, if there are two or three fixed points, then the Chern class map is somewhere injective. Since $\dim M \geq 6$, by Theorem \ref{t27}, there must be at least four fixed points. \end{proof}

With all of the above, we are ready to classify an almost complex $S^1$-manifold with three fixed points.

\begin{theorem} \label{t29} Let the circle act on a compact, connected almost complex manifold $M$. If there are exactly three fixed points, then $\dim M=4$. Moreover, the weights at the fixed points are $\{a+b,a\},\{-a,b\}$, and $\{-b,-a-b\}$ for some positive integers $a$ and $b$. \end{theorem}

For the rest of the paper, we discuss the proof of Theorem \ref{t29}. As mentioned in the introduction, the idea of the proof is to adapt the proof of Theorem \ref{t11} in \cite{J1} with eliminating the reliance on the symplectic property. The proof of Theorem \ref{t11} directly uses the following results by others:
\begin{enumerate}
\item In \cite{Ka}: Theorem 4.1.
\item In \cite{MD}: Proposition 2.
\item In \cite{PT}: Theorem 3, Corollary 4, Lemma 11, Corollary 12, and Lemma 13.
\item In \cite{T}: Lemma 2.6.
\end{enumerate}

Note that Theorem \ref{t27}, Theorem \ref{t24}, Theorem \ref{t26}, Corollary \ref{c28}, Lemma \ref{l22}, Corollary \ref{c25}, and Lemma \ref{l23} extend Theorem 1, Theorem 2, Theorem 3, Corollary 4, Lemma 11, Corollary 12, and Lemma 13 of \cite{PT} from symplectic $S^1$-actions to almost complex $S^1$-manifolds, respectively. Moreover, Lemma \ref{l24} extends Lemma 2.6 of \cite{T}.

On the other hand, to prove Theorem \ref{t29}, we do not need Theorem 4.1 of \cite{Ka} and Proposition 2 of \cite{MD}. Theorem 4.1 of \cite{Ka} classifies 4-dimensional compact Hamiltonian $S^1$-spaces by their associated multigraphs, up to equivariant symplectomorphism. Proposition 2 of \cite{MD} proves that a symplectic circle action on a 4-dimensional compact symplectic manifold is Hamiltonian if and only if there is a fixed point.

Therefore, we have extended all the results that we need to prove Theorem \ref{t29} from the symplectic case to the almost complex case.

We shall give a brief outline of the proof of Theorem \ref{t29}. To prove Theorem \ref{t29}, we use induction on the dimension of the manifold $M$. The key idea is to get restrictions on the weights at the three fixed points. We prove that if $\dim M>4$, then the weights at the fixed points cannot satisfy all the restrictions, which means that such a manifold $M$ cannot exist.

By quotienting out by the subgroup which acts trivially, we may assume that the action is effective. As mentioned at the beginning of this section, for an integer $k>1$, the subgroup $\mathbb{Z}_k$ acts on $M$ and there is an induced circle action on each component of $M^{\mathbb{Z}_k}$. By the inductive hypothesis, we only have a few possibilities for $M^{\mathbb{Z}_k}$; for details see Lemma 4.5 of \cite{J1}. This is one of the main ideas of the proof of Theorem \ref{t11}.

Let $\dim M=2n$. By Corollary \ref{c25}, $n$ is even. By using the same argument as in Proposition 3.2 of \cite{J1}, the largest weight occurs only once among all the weights at the fixed points, and so does the smallest weight. Using this, we also show that the numbers of negative weights at the fixed points are $\frac{n}{2}-1$, $\frac{n}{2}$, and $\frac{n}{2}+1$ (for details, see Lemma 3.4 of \cite{J1}; twice of the number of negative weights at a fixed point is called the \emph{index} of the fixed point in \cite{J1}.). In fact, one can prove this by combining Lemma \ref{l22} and Lemma \ref{l25}.

Suppose that $\dim M=4$. By Lemma \ref{l23}, there exist positive integers $a,b$, and $c$ such that the weights at the fixed points are $\{a,c\}, \{-a,b\}$, and $\{-b,-c\}$. As in the proof of Proposition 4.3 of \cite{J1}, from push-forward of 1 in ABBV localization theorem, we have that $c=a+b$. This proves the case where $\dim M=4$.

Suppose $\dim M>4$. Then Corollary \ref{c28} says that the first equivariant Chern class at each fixed point is zero. By considering possible submanifolds for $M^{\mathbb{Z}_2}$ with this fact and the inductive hypothesis, we show that $M^{\mathbb{Z}_2}$, the set of points fixed by the $\mathbb{Z}_2$-action, is a 4-dimensional almost complex submanifold that contains all the fixed points. Moreover, we determine the weights at the fixed points; see Lemma 4.4 of \cite{J1}.

For technical reasons, we separate into two cases; the case that the largest weight is odd (section 5 of \cite{J1}) and the case that the largest weight is even (section 6-8 of \cite{J1}). In section 5 of \cite{J1} we prove that if $\dim M>4$, then the largest weight cannot be odd. In sections 6-8 we prove that if $\dim M>4$, then the largest weight cannot be even. Therefore, we conclude that $\dim M=4$ and hence the theorem follows. No problem occurs here at all to copy down the proofs in sections 5-8 of \cite{J1} to prove Theorem \ref{t29} with changing the word symplectic manifold to almost complex manifold and the word symplectic circle action to circle action that preserves the almost complex structure.

\begin{proof}[Proof of Theorem \ref{t29}] Apply the proof of Theorem \ref{t11} in \cite{J1} with the following modifications.
\begin{enumerate}[(1)]
\item Replace a symplectic (sub) manifold with an almost complex (sub) manifold.
\item Replace a symplectic circle action with a circle action on an almost complex manifold that preserves the almost complex structure.
\item Replace Lemma 2.2, Corollary 2.3, Lemma 2.4, Lemma 2.5, Theorem 2.6, and Corollary 2.7 of \cite{J1} with Lemma \ref{l22}, Corollary \ref{c25}, Lemma \ref{l23}, Lemma \ref{l24}, Theorem \ref{t26}, and Corollary \ref{c28}.
\item In Proposition 4.3 of \cite{J1}, replace the conclusion `If $\dim M<8$, then $M$ is equivariantly symplectomorphic to $\mathbb{CP}^2$' with `If $\dim M<8$, then $\dim M=4$ and the weights at the fixed points are $\{a+b,a\}$, $\{-a,b\}$, and $\{-b,-a-b\}$ for some positive integers $a$ and $b$.'
\item Ignore the parts that use Proposition 4.1 and Theorem 4.2 in \cite{J1}.
\end{enumerate}
 \end{proof}

\begin{rem} Consider a symplectic circle action on a compact symplectic manifold $M$. As in \cite{PT}, if there are two fixed points, then either $M$ is the 2-sphere or $\dim M=6$. However, we do not know if there exists such a manifold $M$ with $\dim M=6$, and the existence or non-existence of such a manifold is an important question for the classification problem of symplectic $S^1$-actions. Note that if such a manifold exists, then the action cannot be Hamiltonian, because any compact Hamiltonian $S^1$-space $M$ has at least $\frac{1}{2}\dim M+1$ fixed points.

On the other hand, as mentioned in the introduction, there exists a circle action on a 6-dimensional compact almost complex manifold with two fixed points, that is, a rotation on the 6-sphere.

Since in the proof of Theorem 1.1 of \cite{J1} the author allows the possibility that there is a symplectic circle action on a 6-dimensional compact symplectic manifold with two fixed points (which does exist in the case of almost complex $S^1$-manifolds), no problem occurs in extending Theorem \ref{t11} to Theorem \ref{t29}. \end{rem}


\begin{thebibliography}{1}

\bibitem[CHK]{CHK}
J. Carrell, A. Howard, and C. Kosniowski: \emph{Holomorphic vector fields on complex surfaces.} Math. Ann. \textbf{204} (1973), 73-81.

\bibitem[GS]{GS}
L. Godinho and S. Sabatini: \emph{New tools for classifying Hamiltonian circle actions with isolated
fixed points}, Found. Comput. Math. \textbf{14} (2014), 791-860.

\bibitem[H]{H}
A. Hattori: \emph{$S^1$-actions on unitary manifolds and quasi-ample line bundles}, J. Fac. Sci. Univ. Tokyo Sect. IA Math. \textbf{31} (1985), 433-486.

\bibitem[JT]{JT}
D. Jang and S. Tolman: \emph{Hamiltonian circle actions on eight dimensional manifolds with minimal fixed sets.} Transformation Groups (2016) doi:10.1007/s00031-016-9370-0.

\bibitem[J1]{J1}
D. Jang: \emph{Symplectic periodic flows with exactly three equilibrium points}, Ergodic Theory and Dynamical Systems \textbf{34} (2014), 1930-1963.

\bibitem[J2]{J2}
D. Jang: \emph{Symplectic circle actions with isolated fixed points.} arXiv:1412.4169, to appear in Journal of Symplectic Geometry.

\bibitem[J3]{J3}
D. Jang: \emph{Circle actions on almost complex manifolds with 4 fixed points.} arXiv:1701.08238.

\bibitem[J4]{J4}
D. Jang: \emph{Circle actions on four-dimensional oriented manifolds with discrete fixed point sets}. arXiv:1703.05464.

\bibitem[Ka]{Ka}
Y. Karshon: \emph{Periodic Hamiltonian flows on four dimensional manifolds.} Mem. Amer.
Math. Soc. \textbf{672} (1999). MR 2000c:53113.

\bibitem[Ko1]{K}
C. Kosniowski: \emph{Holomorphic vector fields with simple isolated zeros}. Math. Ann. \textbf{208} (1974), 171-173.

\bibitem[Ko2]{K2}
C. Kosniowski: \emph{Some formulae and conjectures associated with circle actions}. Topology
Symposium, Siegen 1979 (Proc. Sympos., Univ. Siegen, Siegen, 1979), 331-339, Lecture
Notes in Math., 788, Springer, Berlin, 1980.

\bibitem[L1]{L}
P. Li: \emph{The rigidity of Dolbeault-type operators and symplectic circle actions.} Proc. Amer. Math. Soc. \textbf{251} (2011), 1987-1995.

\bibitem[L2]{L2}
P. Li :\emph{Circle action with prescribed number of fixed points.} Acta Mathematica Sinica, English Series, June 2015, Volume 31, Issue 6, pp 1035-1042.

\bibitem[MD]{MD}
D. McDuff: \emph{The moment map for circle actions on symplectic manifolds.} J. Geom. Phys. \textbf{5} (1988), no. 2, 149-160.

\bibitem[PT]{PT}
A. Pelayo and S. Tolman: \emph{Fixed points of symplectic periodic flows.} Ergodic Theory and Dynamical Systems \textbf{31} (2011), 1237-1247.

\bibitem[T]{T}
S. Tolman: \emph{On a symplectic generalization of Petrie's conjecture.} Trans. Amer. Math. Soc. \textbf{362} (2010), no.8, 3963-3996. MR2638879.
\end{thebibliography}
\end{document}